\documentclass[12pt,twoside]{amsart}
\usepackage{amsthm,amsfonts,amssymb,amscd,euscript}

\usepackage[matrix,arrow]{xy}

\author{Florin Ambro} 
\address{Institute of Mathematics ``Simion Stoilow'' of the Romanian
Academy\\
P.O. BOX 1-764, RO-014700 Bucharest\\ 
Romania.}
\email{florin.ambro@imar.ro}

\newcommand{\isoto}{{\overset{\sim}{\rightarrow}}}

\newcommand{\Q}{{\mathbb Q}}
\newcommand{\Z}{{\mathbb Z}}
\newcommand{\N}{{\mathbb N}}
\newcommand{\R}{{\mathbb R}}

\newcommand{\bP}{{\mathbb P}} 
\newcommand{\bA}{{\mathbb A}} 

\newcommand{\cA}{{\mathcal A}}

\newcommand{\cF}{{\mathcal F}}
\newcommand{\cH}{{\mathcal H}}
\newcommand{\cI}{{\mathcal I}}
\newcommand{\cK}{{\mathcal K}}

\newcommand{\cL}{{\mathcal L}}
\newcommand{\cM}{{\mathcal M}}
\newcommand{\cO}{{\mathcal O}}

\newcommand{\Cond}{\operatorname{Cond}}

\newcommand{\Gr}{\operatorname{Gr}}

\newcommand{\LCS}{\operatorname{LCS}}

\newcommand{\relint}{\operatorname{relint}}
\newcommand{\Res}{\operatorname{Res}}

\newcommand{\Sing}{\operatorname{Sing}}

\newcommand{\Spec}{\operatorname{Spec}}
\newcommand{\Supp}{\operatorname{Supp}}

\theoremstyle{plain}
\newtheorem{thm}{Theorem}[section]

\newtheorem{lem}[thm]{Lemma}
\newtheorem{cor}[thm]{Corollary}

\newtheorem{prop}[thm]{Proposition}

\theoremstyle{definition}
\newtheorem{defn}[thm]{Definition}

\newtheorem{exmp}[thm]{Example}

\newtheorem{rem}[thm]{Remark}

\newtheorem{ack}{Acknowledgments}   

\theoremstyle{remark}

\begin{document}

\bibliographystyle{amsalpha+}
\title{An injectivity theorem II}
\maketitle

\begin{abstract} 
We extend the injectivity theorem of Esnault and Viehweg to a class of non-normal
log varieties, which contains normal crossings log varieties, and is closed under
the operation of taking the $\LCS$ locus.
\end{abstract} 



\footnotetext[1]{2010 Mathematics Subject Classification. Primary: 14F17 Secondary: 14E30.}

\footnotetext[2]{Keywords: normal crossings singularities, residues, vanishing theorems.}


\section*{Introduction}


The birational classification of complex manifolds rests 
on vanishing theorems for Cartier divisors of the form $L\sim_\Q K_X+B$, where 
$(X,B)$ is a {\em log smooth variety} (i.e. $X$ is a smooth complex variety and $B=\sum_ib_iE_i$
is a boundary with coefficients $b_i\in \Q\cap [0,1]$, such that $\sum_iE_i$ is a
normal crossings divisor on $X$). In the order in which one may prove these vanishing
theorems, they are Esnault-Viehweg injectivity,
Tankeev-Koll\'ar injectivity, Koll\'ar's torsion freeness, Ohsawa-Koll\'ar vanishing, 
Kawamata-Viehweg vanishing. The injectivity theorems imply the rest. 
Modulo cyclic covering tricks and Hironaka's desingularization, 
the injectivity theorems are a direct consequence of the $E_2$-degeneration of the Hodge to de Rham 
spectral sequence associated to an open manifold.

To study the category of log smooth varieties, it is necessary to enlarge it to allow certain 
non-normal, even reducible objects, which appear in inductive arguments in the study of linear systems, 
or in compactification problems for moduli spaces of manifolds. The smallest such enlargement 
is the category of {\em normal crossings log varieties} $(X,B)$, which may be thought as glueings of log
smooth varieties, in the simplest possible way. By definition, they are locally analytically isomorphic
to the local model $0\in X=\cup_{i\in I}\{z_i=0\}\subset \bA^N$, endowed with
the boundary $B=\sum_{j\in J} b_j \{z_j=0 \}|_X$, where $I,J$ are disjoint subsets of $\{1,\ldots,N\}$
and $b_j\in \Q\cap [0,1]$. Since $X$ has Gorenstein singularities, the dualizing sheaf $\omega_X$ 
is an invertible $\cO_X$-module. We denote by $K_X$ a Cartier divisor on $X$ such that $\omega_X\simeq \cO_X(K_X)$.
By definition, $B$ is $\Q$-Cartier. Normal crossings varieties are build up of their lc centers, closed 
irreducible subvarieties, which on the local analytic model correspond to (unions of) affine subspaces 
$\cap_{i\in I'}\{z_i=0\} \cap \cap_{j\in J'}\{z_j=0\}\subset \bA^N$, where $I'\subset I$ is a non-empty
subset, and $J'\subset \{j\in J;b_j=1\}$ is a possibly empty subset. For example, the irreducible
components of $X$ are lc centers of $(X,B)$. Inside the category of normal crossings log varieties,
log smooth varieties are exactly those with normal ambient space. The aim of this paper is to show 
that the above mentioned vanishing theorems remain true in the category of normal crossings varieties.

\begin{thm}\label{MainNC}
Let $(X,B)$ be a normal crossings log variety, $L$ a Cartier divisor on $X$, and $f\colon X\to Y$ 
a proper morphism.
\begin{itemize}
\item[1)] (Esnault-Viehweg injectivity) Suppose $L\sim_\Q K_X+B$. Let $D$ be an effective Cartier 
divisor supported by $B$. Then the natural homomorphisms $R^qf_* \cO_X(L) \to R^qf_* \cO_X(L+D)$ 
are injective.
\item[2)] (Tankeev-Koll\'ar injectivity) Suppose $L\sim_\Q K_X+B+H$, where $H$
is a semiample $\Q$-divisor. Let $D$ be an effective Cartier divisor which contains 
no lc center of $(X,B)$, and such that $D\sim_\Q u H$ for some $u>0$. 
Then the natural homomorphisms $R^qf_*\cO_X(L)\to R^qf_*\cO_X(L+D)$ are injective.
\item[3)] (Koll\'ar's torsion freeness) Suppose $L\sim_\Q K_X+B$. Let $s$ be a local section of $R^qf_*\cO_X(L)$ 
whose support does not contain $f(C)$, for every lc center $C$ of $(X,B)$. Then $s=0$.
\item[4)] (Ohsawa-Koll\'ar vanishing) Let $g\colon Y \to Z$ be a projective morphism. Suppose $L\sim_\Q K_X+B+f^*A$, 
where $A$ is a $g$-ample $\Q$-Cartier divisor on $Y$. Then $R^pg_*R^qf_*\cO_X(L)=0$ for $p\ne 0$.
\end{itemize}
\end{thm}

The notation $L\sim_\Q M$ means that there exists a positive integer $r$ such that both $rL$ and $rM$ are Cartier 
divisors, and $\cO_X(rL)\simeq \cO_X(rM)$. Kawamata-Viehweg vanishing is the case $\dim Z=0$ of the Ohsawa-Koll\'ar
vanishing.

Theorem~\ref{MainNC}.2)-4) was proved by Kawamata~\cite{Kaw85} if $B$ has coefficients 
strictly less than $1$, and it was proved for {\em embedded} normal crossings varieties $(X,B)$ in~\cite[Section 3]{Amb03}.
We remove the global embedded assumption in this paper, as expected in~\cite[Remark 2.9]{Amb03}.
Theorem~\ref{MainNC}.1) is implicit in the proof of~\cite[Theorem 3.1]{Amb03}, in the case when $(X,B)$ is 
embedded normal crossings and $D$ is supported by the part of $B$ with coefficients strictly less than $1$, 
which is the original setting of Esnault and Viehweg. We observed in~\cite{Amb14} that the same results holds if $D$ is 
supported by $B$, and Theorem~\ref{MainNC}.1) extends~\cite{Amb14} to the normal crossings case.
 
Theorem~\ref{MainNC} is proved by reduction to the log smooth case. There are two known methods of proof.
Let $\bar{X}\to X$ be the normalization, let $X_n=(\bar{X}/X)^{n+1}$ for $n\ge 0$. With the natural projections
and diagonals, we obtain a simplicial algebraic variety $X_\bullet$, together with a natural augmentation 
$\epsilon \colon X_\bullet\to X$. The key point is that each $X_n$ is smooth,
so we may really think of $\epsilon$ as a resolution of singularities. The method in~\cite{Kaw85} is to
use the descent spectral sequence to deduce a statement on $X$ from the same statement on each $X_n$.
The method in~\cite{Amb03} is to lift the statement from $X$ to a statement on $X_\bullet$, and imitate the 
proof used in the log smooth case in this simplicial setting. In this paper we use the method in~\cite{Kaw85}.
The new idea is an adjunction formula
$$
(K_X+B)|_{X_n}\sim_\Q K_{X_n}+B_n,
$$ 
for a suitable log smooth structure $(X_n,B_n)$, for each $n$. Moreover, $(X_n,B_n)$ glue to a log smooth simplicial 
variety. To achieve this, we observe that each irreducible component of $X_n$ is the normalization of some lc center 
of $(X,B)$. Then the adjunction formula follows from the theory of residues for normal crossings varieties developed 
in~\cite{TFR2}. To construct residues for normal crossings varieties we have to deal with slightly more general 
singularities, namely{\em generalized normal crossings log varieties}. The motivation for this enlargement,
is that if $X$ has normal crossings singularities, then $\Sing X$ may not have normal crossings singularities. But
if $X$ has generalized normal crossings singularities, so does $\Sing X$. We actually prove Theorem~\ref{MainNC}
in the category of generalized normal crossings singularities (Theorems~\ref{EVGNC},~\ref{TKGNC},~\ref{TFGNC},~\ref{OKGNC}). 
The same proof works in the category of normal crossings log varieties, provided their residues to lc centers are
taken for granted. Note that generalized normal crossings
singularities in our sense are more general than those defined by Kawamata~\cite{Kaw85}.
For example, every seminormal curve is generalized normal crossings.

To illustrate how generalized normal crossings appear, let us consider two examples of residues.
First, consider the log smooth variety $(\bA^2,H_1+H_2)$, where $H_1,H_2$ are the standard 
hyperplanes, intersecting at the origin $0$. We want to perform adjunction from $(\bA^2,H_1+H_2)$
to its lc center $0$. We may first take residue onto $H_1$, and end up with the log structure $(H_1,0)$,
and then take residue from $(H_1,0)$ to $0$. But we may also restrict to $(H_2,0)$, and then to $0$.
The two chains of residues do not coincide; they differ by $-1$. Since an analytic isomorphism interchanges
the two hyperplanes, none of the above compositions of residues is canonical. But they become canonical if
raised to even powers. We obtain a canonical residue isomorphism
$$
\Res^{[2]}_{\bA^2\to 0}\colon \omega_{\bA^2}(\log H_1+H_2)^{\otimes 2}|_0 \isoto \omega_0^{\otimes 2}
$$
Now we construct the same residue isomorphism, without coordinates.
Denote $C=H_1+H_2$. Let $\omega_C$ be the sheaf whose sections are rational differential forms
which are regular outside $0$, and on the normalization $H_1\sqcup H_2$ of $C$ induce forms with
logarithmic poles along the two points $O_1,O_2$ above the origin, and have the same residues at $O_1,O_2$.
One checks that $\omega_C$ is an invertible $\cO_C$-module. The residues from $\bA^2$ to the irreducible components 
of normalization of $C$ glue to a residue isomorphism 
$$
\Res^{[2]}_{\bA^2\to C}\colon \omega_{\bA^2}(\log C)^{\otimes 2}|_C \isoto \omega_C^{\otimes 2}.
$$
Since the forms of $\omega_C$ have the same residues above the origin, we also obtain a residue isomorphism
$$
\Res^{[2]}_{C\to 0}\colon \omega_C^{\otimes 2}|_0 \isoto \omega_0^{\otimes 2}.
$$
The composition $\Res^{[2]}_{C\to 0}\circ \Res^{[2]}_{\bA^2\to C}$ is exactly $\Res^{[2]}_{\bA^2\to 0}$. It is intrinsic,
independent of the choice of coordinates, or analytic isomorphisms. Note that $\omega_C$ differs from the Rosenlicht dualizing sheaf 
$\Omega_C$, but $\omega_C^{\otimes m}=\Omega^{\otimes m}_C$ for $m\in 2\Z$ (at the origin, the local generator 
for $\omega_C$ is $(\frac{dz_1}{z_1},\frac{dz_2}{z_2})$, and for $\Omega_C$ is $(\frac{dz_1}{z_1},-\frac{dz_2}{z_2})$).

Second, let $S$ be the normal crossings surface $(xyz=0)\subset \bA^3$, set $B=0$. We want to perform adjunction from
$S$ to its lc center the origin. As above, we may first restrict to a plane, then to a line, and then to the origin. There are several
choices of chains, which coincide up to a sign. If we raise to an even power, we obtain residue isomorphisms from $S$ to $0$. These
are invariant under analytic isomorphisms, since we can also define them in the following invariant way. Let $C=\Sing S$.
Then $C$ is the union of coordinate axis in $\bA^3$, a seminormal curve which is not Gorenstein. The usual dualizing sheaf
is useless in this situation. We may define $\omega_C$ as above (requiring same residues over the origin), and then
$\omega_C$ is an invertible $\cO_C$-module (at the origin, the local generator is $(\frac{dz_1}{z_1},\frac{dz_2}{z_2},\frac{dz_3}{z_3})$),
and residues from $S$ to the irreducible components of the normalization of $C$ glue to a residue isomorphism 
$$
\Res^{[2]}_{S \to C}\colon \omega_S^{\otimes 2}|_C \isoto \omega_C^{\otimes 2}.
$$
The singular locus of $C$ is $0$, and we again obtain a residue isomorphism 
$$
\Res^{[2]}_{C\to 0}\colon \omega_C^{\otimes 2}|_0 \isoto \omega_0^{\otimes 2}.
$$
The composition $\Res^{[2]}_{C\to 0}\circ \Res^{[2]}_{S\to C}$ is exactly $\Res^{[2]}_{S\to 0}$, defined from coordinates.

The conclusion we draw from these two examples is that we must redefine the powers of log canonical sheaf 
$\omega^{[n]}_{(X,B)}\ (n\in \Z)$ (without dualizing property),
and we must allow singularities which are not normal crossings, but very close. In~\cite{TFR2}, we constructed residues for
so called n-wlc varieties. Generalized normal crossings varieties are a special case of n-wlc varieties.

We outline the structure of this paper. In Section 1, we construct the simplicial log variety induced by a n-wlc log variety.
The reader should be familiar with~\cite[Sections 3 and 5]{TFR2}. In Section 2, we define generalized normal crossings log varieties, and 
analyze the induced simplicial log variety. In Section 3, we prove the vanishing theorems. The injectivity theorems are 
reduced to the smooth case, using the simplicial log structure induced. The torsion freeness and vanishing theorems are
deduced then by standard arguments. In Section 4, we collect some inductive properties of generalized normal crossings
varieties. The key inductive property is that the $\LCS$-locus of a generalized normal crossings log variety is again a 
generalized normal crossings log variety, for a suitable boundary, and we can perform adjunction onto the $\LCS$-locus.
We hope that in the future one may be able to use these inductive properties to reprove the vanishing theorems in Section 3.

\begin{ack} 
This work was mostly done during a visit to the IBS Center for Geometry and Physics in Pohang, Korea. 
I am grateful to Jihun Park for hospitality.
\end{ack}


\section{Preliminary}


All varieties are defined over an algebraically closed field $k$, of characteristic zero.

A {\em log smooth variety} is a pair $(X,B)$, where $X$ is a smooth $k$-variety and $B=\sum_ib_iE_i$
is a boundary such that $b_i\in \Q\cap [0,1]$ and $\sum_iE_i$ is a NC divisor.

We refer the reader to~\cite{TFR2} for the definition and basic properties of wlc varieties $(X/k,B)$, and
some special cases:  toric and n-wlc. We will remove the fixed ground field $k$ from notation; for example we 
denote $\omega^{[n]}_{(X/k,B)}$ by $\omega^{[n]}_{(X,B)}$.

\begin{lem}\label{ilc}
Let $(X',B_{X'})$ and $(X,B)$ be normal log pairs, let $f\colon (X',B_{X'})\to (X,B)$ be \'etale and
log crepant. Let $Z' \subset X'$ be a closed irreducible subset. Then $Z'$ is an lc center of $(X',B_{X'})$ if
and only if $f(Z')$ is an lc center of $(X,B_X)$. 
\end{lem}

\begin{proof}
Cutting $f(Z')$ with general hyperplane sections, we may suppose $Z'$ is a closed point $P'$.
Since $f$ is open, we may replace $X$ by the image of $f$ and suppose $f$ is surjective. 
After removing from $X'$ the finite set $f^{-1}f(P')\setminus P'$, we may also suppose $f^{-1}f(P')=P'$.
Then the claim follows from~\cite[page 46, 2.14.(2)]{Kbook}.
\end{proof}


\subsection{Simplicial log structure induced by a n-wlc log variety}


Let $(X,B)$ be a n-wlc log variety (see~\cite[Section 5]{TFR2}). Let $r\in (2\Z)_{>0}$ such that $rB$ has integer
coefficients and $\omega^{[r]}_{(X,B)}$ is an invertible $\cO_X$-module.
Let $\pi\colon \bar{X}\to X$ be the normalization.
Then $X_n=(\bar{X}/X)^{n+1}\ (n\ge 0)$ are the components of a simplicial $k$-algebraic
variety $X_\bullet$, endowed with a natural augmentation $\epsilon\colon X_\bullet\to X$.

\begin{prop}\label{slv} The following properties hold:
\begin{itemize}
\item[a)] Each $X_n$ is normal. Let $Z_n$ be an irreducible component of $X_n$. Then $\epsilon_n
\colon Z_n\to X$ is the normalization of an lc center of $(X,B)$. Let $(Z_n,B_{Z_n})$
be the n-wlc log variety structure induced by the residue isomorphism
$
\Res^{[r]}\colon \omega^{[r]}_{(X,B)}|_{Z_n}\isoto \omega^{[r]}_{(Z_n,B_{Z_n})}
$
(see~\cite[Theorem 5.9]{TFR2}).

Let 
$
(X_n,B_n)=\sqcup_{Z_n}(Z_n,B_{Z_n})
$ 
be the induced structure of normal log variety, with n-wlc singularities (independent of the choice of $r$). 
We obtain isomorphisms
$$
\Res^{[r]}_{X\to X_n}\colon \epsilon_n^*\omega^{[r]}_{(X,B)} \isoto \omega^{[r]}_{(X_n,B_n)}.
$$
Moreover, each lc center of $(X,B)$ is the image of some lc center of $(X_n,B_n)$.

\item[b)] Let $\varphi\colon X_m\to X_n$ be the simplicial morphism induced by an
order preserving morphism $\Delta_n\to \Delta_m$, for some $m,n\ge 0$. It induces a commutative diagram
\[ 
\xymatrix{
   X_m  \ar[dr]_{\epsilon_m} \ar[rr]^\varphi     &        &  X_n    \ar[dl]^{\epsilon_n}  \\
     &      X       &  
} \]
Let $Z_m$ be an irreducible component of $X_m$. Then $\varphi \colon Z_m\to X_n$
is the normalization of an lc center of $(X_n,B_n)$. Let 
$
\Res^{[r]}\colon \omega^{[r]}_{(X_n,B_n)}|_{Z_m}\isoto \omega^{[r]}_{(Z_m,B_{Z_m})}
$
be the induced residue isomorphism. Let $\Res^{[r]}_\varphi\colon \varphi^* \omega^{[r]}_{(X_n,B_n)}
\isoto \omega^{[r]}_{(X_m,B_m)}$ be the induced isomorphism. Then 
$$
\Res^{[r]}_\varphi\circ \ \varphi^*\Res^{[r]}_{X\to X_n}=\Res^{[r]}_{X\to X_m}.
$$
In particular, $\omega^{[r]}_{(X_n,B_n)}$ and $\Res^{[r]}_\varphi$
form an $\cO_{X_\bullet}$-module $\omega^{[r]}_{(X_\bullet,B_\bullet)}$, 
endowed with an isomorphism
$
\epsilon^*\omega^{[r]}_{(X,B)} \isoto \omega^{[r]}_{(X_\bullet,B_\bullet)}.
$
\end{itemize}
\end{prop}

\begin{proof}
By~\cite{Ar69}, we may suppose $(X,B)$ coincides with a local model.
That is $X=\Spec k[\cM]$ is the toric variety associated with a monoidal
complex $\cM=(M,\Delta,(S_\sigma)_{\sigma\in \Delta})$, $X$ has normal
irreducible components, $B$ is an effective boundary supported by invariant
prime divisors at which $X$ is smooth, and $(X,B)$ has wlc singularities.

Let $X=\cup_F X_F$ be the decomposition into irreducible components,
where the union runs after all facets $F$ of $\Delta$. Let $\psi\in \cap_F \frac{1}{r}S_F$ be
the log discrepancy function of $(X,B)$. By assumption, each irreducible component
$X_F$ is normal. Therefore $\bar{X}=\sqcup_F X_F$. We obtain 
$$
X_n=\sqcup_{F_0,\ldots,F_n}X_{F_0\cap \cdots \cap F_n}.
$$
Since $\psi\in F_0\cap \cdots \cap F_n$, each $X_{F_0\cap \cdots \cap F_n}$
is an lc center of $(X,B)$. The toric log structure induced via residues on 
$X_{F_0\cap \cdots \cap F_n}$ is that induced by the log discrepancy function 
$\psi\in F_0\cap \cdots \cap F_n$. 
 
An lc center of $(X,B)$ is of the form $X_\gamma$, with $\psi\in \gamma\in \Delta$.
If $F$ is a facet of $\Delta$ which contains $\gamma$, then $X_\gamma$ is also 
an lc center of the irreducible component $(X_{F\cap \cdots \cap F},B_n)$ of 
$(X_n,B_n)$. This proves a).

For b), recall that any simplicial morphism is a composition of face morphisms
$\delta_i\colon X_{n+1}\to X_n$ and degeneracy morphisms $s_i\colon X_n\to X_{n+1}$.
Hence suffices to verify b) for face and degeneracy morphisms. For our local model,
$\delta_i$ embeds $X_{F_0\cap \cdots \cap F_{n+1}}$ into $X_{F_0\cap \cdots \widehat{F_i} \cdots\cap F_{n+1}}$,
and $s_i$ maps $X_{F_0\cap \cdots \cap F_n}$ isomorphically onto 
$X_{F_0\cap \cdots \cap F_i\cap F_i\cap \cdots \cap F_n}$. Then b) holds in our case,
since all log structures involved have the same log discrepancy function $\psi$.
\end{proof}


\section{GNC log varieties}


Recall first some standard notation. The set $\{1,2,\ldots,N\}$ is denoted by $[N]$,
the $k$-affine space $\bA^N_k$ has coordinates $(z_i)_{i\in [N]}$, and $H_i=\{z\in \bA^N;z_i=0\}$
is the standard $i$-th hyperplane. For a subset $F\subseteq [N]$, denote 
$\bA_F=\cap_{i\in [N]\setminus F} \{z\in \bA^N; z_i=0\}$. It is an affine space with coordinates
$(z_i)_{i\in F}$.

\begin{defn}
A {\em GNC (generalized normal crossings) local model} is a pair $(X,B)$, of the following form:
\begin{itemize}
\item[a)] $X=\cup_F \bA_F\subset \bA_k^N$, where the union is indexed after finitely many subsets 
$F\subseteq [N]$ (called facets), not contained in one another. We assume 
$X$ satisfies Serre's property $S_2$, that is for any two facets $F\ne F'$, there exists a chain of facets $F=F_0,F_1,\ldots,F_l=F'$
such that for every $0\le i<l$, $F_i\cap F_{i+1}$ contains $F\cap F'$ and it has codimension one in both
$F_i$ and $F_{i+1}$. 
\item[b)] Denote $\sigma=\cap_F F$. If $\sigma\prec \tau\prec F$ and $\tau$ has codimension one in $F$,
then there exists a facet $F'$ such that $\tau=F\cap F'$.
\item[c)] $B=(\sum_{i\in \sigma}b_iH_i)|_X$, where $b_i \in \Q\cap [0,1]$ and $H_i=\{z\in \bA^N; z_i=0\}$.
We may rewrite $B=\sum_F\sum_{i\in \sigma}b_i  \bA_{F\setminus i}$.
\end{itemize}
\end{defn}

We claim that $(X,B)$ is a toric wlc log variety. 
Note first that $X$ is the toric variety $\Spec k[\cM]$ associated to the monoidal complex
$\cM=(M,\Delta,(S_\sigma)_{\sigma\in \Delta})$, where $M=\Z^N$, $\Delta$ is the fan consisting 
of the cones $\sum_{i\in F} \R_{\ge 0}m_i$ and all their faces, and $S_\sigma=\Z^N\cap \sigma$ for $\sigma\in \Delta$.
Here $m_1,\ldots,m_N$ denotes the standard basis of the semigroup $\N^N$.
Each irreducible component of $X$ is smooth. The normalization of $X$ is 
$\bar{X}=\sqcup_F \bA_F$. Denote $\psi=\sum_{i\in \sigma}(1-b_i)m_i$. 
On $\bA_F$, $\psi$ induces the log structure with boundary
$$
B_{\bA_F}=\sum_{i\in F\setminus \sigma}\bA_{F\setminus i}+\sum_{i\in \sigma}b_i\bA_{F\setminus i}.
$$
Let $\bar{C}\subset \bar{X}$ be the conductor subscheme.
By a), $\bar{C}|_{\bA_F}\le \sum_{i\in F\setminus \sigma}\bA_{F\setminus i}$. Equality holds if and only
if b) holds. Therefore 
$$
B_{\bA_F}=\bar{C}|_{\bA_F}+\sum_{i\in \sigma}b_i\bA_{F\setminus i}=(\bar{C}+\bar{B})|_{\bA_F}.
$$

We conclude that the irreducible components of $(\bar{X},\bar{C}+\bar{B})$ have the same log discrepancy
function $\psi$, and therefore $(X,B)$ is a toric wlc log variety, by~\cite[Proposition 4.10]{TFR2}. Note that $X$
is $\Q$-orientable by~\cite[Lemma 4.7 and Example 4.8.(2)]{TFR2}.
If $2\mid r$ and $rb_i\in \Z$ for all $i\in \sigma$, then 
$\omega^{[r]}_{(X,B)}\simeq \cO_X$. Given a), properties b) and c) are equivalent to 
\begin{itemize}
\item[b')] $(X,0)$ is a toric wlc log variety.
\item[c')] $B$ is a torus-invariant boundary whose support contains no lc center of $(X,0)$.
\end{itemize}

The $\Q$-divisors $B$, $B^{=1}$, $B^{<1}$ are $\Q$-Cartier (so is the part of $B$ with coefficients 
in a given interval in $\R$).

\begin{exmp} A {\em NC (normal crossings) local model} is a pair $(X,B)$, where 
$X=\cup_{i\in I}H_i\subset \bA_k^N$ and $B=(\sum_{i\notin I}b_iH_i)|_X$,
where $I$ is a non-empty subset of $[N]$ and $b_i\in \Q\cap [0,1]$.
If we set $F=[N]\setminus i\ (i\in I)$, we see that $(X,B)$ is a GNC local model. 
Here we have $\sigma=[N]\setminus I$. 
\end{exmp}

\begin{exmp}
Let $\sigma\subsetneq [N]$, let $|\sigma|\le p<N$.
Let $\{F\}$ consist of all subsets of $[N]$ which have cardinality $p$, and contain $\sigma$. Let $b_i\in \Q \cap [0,1]$,
for $i\in \sigma$. Then $(X=\cup_F\bA_F\subset \bA_k^N,(\sum_{i\in \sigma}b_iH_i)|_X)$ is a GNC local model.
\end{exmp}

\begin{exmp} 
Let $X=\bA_{12}\cup\bA_{23}\cup\bA_{34}\subset \bA_k^4$ and $B=\bA_1+\bA_4$. Then 
$(X,B)$ is a toric wlc log variety (with log discrepancy function $\psi=0$), but not a GNC local model.
\end{exmp}

\begin{defn}
A {\em GNC (NC) log variety} $(X,B)$ is a wlc log variety such that for every closed point $x\in X$, 
there exists a GNC (NC) local model $(X',B')$ and an isomorphism
of complete local $k$-algebras $\cO_{X,x}^\wedge \simeq \cO_{X',0}^\wedge$, such that 
$(\omega^{[r]}_{(X,B)})_x^\wedge$ corresponds to $(\omega^{[r]}_{(X',B')})_0^\wedge$ for $r$ 
sufficiently divisible.
\end{defn}

By~\cite{Ar69}, there exists a common \'etale neighborhood 
$$
\xymatrix{
   & (U,y)  \ar[dl]_i \ar[dr]^{i'} &  \\
(X,x)           &  & (X',0)  
}
$$ 
and a wlc log variety structure $(U,B_U)$ on $U$ such that $i^* \omega^{[n]}_{(X,B)}=
\omega^{[n]}_{(U,B_U)}={i'}^* \omega^{[n]}_{(X',B')}$ for all $n\in \Z$.

It follows that $(X,0)$ is a GNC (NC) log variety, and $B$, $B^{=1}$, $B^{<1}$ are $\Q$-Cartier divisors.

\begin{rem}
Let $(X,B)$ be a NC log variety. Let $\omega_X$ be the canonical choice of dualizing sheaf,
defined by Rosenlicht. It is an invertible $\cO_X$-module, since $X$ is locally complete intersection.
If $rB$ has integer coefficients and $r$ is divisible by $2$, then 
$\omega_X^{\otimes r}\otimes \cO_X(rB)=\omega^{[r]}_{(X,B)}$ (see~\cite{TFR2}).
\end{rem}


\subsection{Simplicial log structure induced by a GNC log variety}

Let $(X,B)$ be a GNC log variety.
Let $\epsilon\colon X_\bullet\to X$ be the simplicial resolution induced by the normalization of $X$.
A GNC log variety is n-wlc. By Proposition~\ref{slv}, residues induce a natural simplicial log variety
structure $(X_\bullet,B_{\bullet})$. In this case $(X_n,B_n)$ is a disjoint union of log smooth
log varieties, and we have residue isomorphisms
$$
\Res^{[r]}_{X\to X_n}\colon \epsilon_n^*\omega^{[r]}_{(X,B)} \isoto \omega^{[r]}_{(X_n,B_n)}
$$
for $r\in (2\Z)_{>0}$ such that $rB$ has integer coefficients.

\begin{lem}\label{slv2} The following properties hold:
\begin{itemize}
\item[1)] $\epsilon\colon X_\bullet\to X$ is a smooth simplicial resolution, and $\cO_X\to R\epsilon_*\cO_{X_\bullet}$
is a quasi-isomorphism.
\item[2)] The lc centers of $(X,0)$ are the images of the irreducible components of $X_n\ (n\ge 0)$.
\item[3)] $(X_n,B_n)$ is a log smooth variety, for all $n$.
\item[4)] The support of $B$ contains no lc center of $(X,0)$, and each $\epsilon_n^*B$ is supported by $B_n$.
\end{itemize}
\end{lem}

\begin{proof} We may suppose $(X,B)$ is a GNC local model. 
Then 
$$
(X_n,B_n)=\sqcup_{F_0,\ldots,F_n}(\bA_{F_0\cap \cdots \cap F_n},
\sum_{i\in F_0\cap\cdots \cap F_n}\bA_{F_0\cap\cdots \cap F_n\setminus i}+
\sum_{i\in \sigma}b_i\bA_{F_0\cap\cdots \cap F_n\setminus i}).
$$

1) Each $X_n$ is smooth, so $\epsilon\colon X_\bullet\to X$ is a smooth simplicial resolution.
By~\cite[Theorem 0.1.b)]{TFR1},  $\cO_X\to R\epsilon_*\cO_{X_\bullet}$ is a quasi-isomorphism.

2) The log variety $(X,0)$ has log discrepancy function $\psi=\sum_{i\in \sigma}m_i\in \relint\sigma$. 
Therefore its lc centers are $X_\gamma$, where $\sigma\prec \gamma\in \Delta$. We claim that 
each such $\gamma$ is an intersection of facets of $\Delta$. Indeed, if $\gamma$ is a facet, the claim 
holds. Else, choose a facet $F$ which contains $\gamma$. Since $\gamma\subsetneq F$, $\gamma$
is the intersection after all codimension one faces $\tau\prec F$ which contain $\gamma$. Each $\tau$
contains the core $\sigma$. Therefore $\tau=F\cap F'$ for some facet $F'$, by axiom b) in the definition of 
GNC local models. We conclude that $\gamma=F_0\cap \cdots \cap F_n$ for some $n\ge 0$.
Therefore $X_\gamma$ appears as an irreducible component of $X_n$.

3) This is clear from the explicit formula for $(X_n,B_n)$.

4) The support of $B$ does not contain the core $X_\sigma$. Since the image on $X$ of an
irreducible component of $X_n$ does contain $X_\sigma$, we obtain that $\epsilon_n^*B$ is
well $\Q$-Cartier defined for all $n$. Moreover, 
$$
(B_n-\epsilon_n^*B)|_{\bA_{F_0\cap \cdots \cap F_n}}=\sum_{i\in F_0\cap\cdots \cap F_n}
\bA_{F_0\cap\cdots \cap F_n\setminus i}.
$$
\end{proof}


\section{Vanishing theorems}


\begin{lem}\label{EVsmooth}
Let $(X,B)$ be a log smooth variety. 
Let $L$ be a Cartier divisor on $X$ such that $L\sim_\Q K_X+B$. Let $D$ be an 
effective Cartier divisor supported by $B$. Let $f\colon X\to Z$ be a proper morphism.
Then the natural homomorphisms
$
R^qf_*\cO_X(L)\to R^qf_*\cO_X(L+D)
$
are injective.
\end{lem}

\begin{proof}
We may suppose $X$ is irreducible, $f$ is surjective, and $Z$ is affine. 
Let $Z\hookrightarrow \bA^N$ be a closed embedding into an affine space.
Compactify $\bA^N\subset \bP^N$ by adding the hyperplane at infinity $H_0$.
Let $Z'\subset \bP^N$ be the closure of $Z$. Let $H=H_0|_{Z'}$. Then $Z\subset Z'$
is an open dense embedding, whose complement $H$ is a hyperplane section.

By Nagata, there exists an open dense embedding $X\subset X''$ such that $X''$ is 
proper. The induced rational map $f\colon X''\dashrightarrow Z'$ is regular on $X$.
By Hironaka's desingularization, there exists a birational contraction 
$X'\to X''$, which is an isomorphism over $X$, such that $X'$ is smooth and $f$ induces
a regular map $f'\colon X'\to Z'$. We may also suppose $\Sigma=X'\setminus X$ is a 
NC divisor, and $(X',B'+\Sigma)$ is log smooth, where $B'=\sum_i b_i(E_i)'$ is the closure 
of $B$ in $X'$ (defined componentwise). We obtained a diagram
\[ 
\xymatrix{
 X \ar[d]_f  \ar[r] & X' \ar[d]^{f'}  \\
 Z \ar[r]  & Z'   
} \]
where the vertical arrows are open dense embeddings, $Z'$ is projective and $X'$ is proper.
The properness of $f$ is equivalent to $X={f'}^{-1}(Z)$, so the diagram is also cartesian.

We represent $L$ by a Weil divisor on $X$. Let $L'$ be its closure in $X'$.
Then $L'\sim_\Q K_{X'}+B'+N$, where $N$ is a $\Q$-divisor supported by $\Sigma$.
Denote $P=L'-\lfloor N\rfloor$ and $\Delta=B'+\{N\}$. Then 
$P\sim_\Q K_{X'}+\Delta$ and $(X',\Delta)$ is log smooth. The closure $D'$ of $D$
in $X'$ is supported by $B'$, hence it is supported by $\Delta$.

Let $m$ be a positive integer. Let $S$ be a general member of the free linear system $| {f'}^*(mH) |$.
Then $P+{f'}^*(mH)\sim_\Q K_{X'}+\Delta+S$, $(X',\Delta+S)$ is log smooth, and $D'$ is supported by 
$\Delta+S$. Denote $\cF=\cO_{X'}(P)$. By~\cite[Theorem 0.1]{Amb14},  the natural homomorphism 
$$
H^n(X',\cF({f'}^*(mH)) )\to H^n(X', \cF({f'}^*(mH)+D')) \ (n\ge 0)
$$
is injective. We have the Leray spectral sequence 
$$
E^{pq}_2=H^p(Z',R^q f'_*\cF(m))\Longrightarrow H^{p+q}(X',\cF( {f'}^*(mH)) ).
$$
Suppose $m$ is sufficiently large. Serre vanishing gives $E^{pq}_2=0$ if $p\ne 0$.
Therefore we obtain a natural isomorphism $H^0(Z',R^n{f'}_*\cF(m))\isoto H^n(X',\cF({f'}^*(mH)) )$.
By the same argument, we have a natural isomorphism 
$H^0(Z',R^n f'_*\cF(D')(m))\isoto H^n(X',\cF(D'+{f'}^*(mH)))$. The injective homomorphism
above becomes the injective homomorphism 
$$
H^0(Z',R^n  f'_*\cF(m)) \to H^0(Z',R^n f'_*\cF(D')(m)).
$$
Since $\cO_{Z'}(m)$ is very ample, this means that $R^n  f'_*\cF \to R^n f'_*\cF(D')$ is injective.
But $X={f'}^{-1}(Z)$, $P|_X=L$, $\cF|_X=\cO_X(L)$ and $D'|_X=D$, so the restriction of this injective 
homomorphism to $Z$ is just the injective homomorphism $R^n f_*\cO_X(L)\to R^n f_*\cO_X(L+D)$.
\end{proof}

\begin{thm}[Esnault-Viehweg injectivity]\label{EVGNC}
Let $(X,B)$ be a GNC log variety. Let $\cL$ be an invertible $\cO_X$-module
such that $\cL^{\otimes r}\simeq \omega^{[r]}_{(X,B)}$ for some $r\ge 1$ such that $rB$ has integer
coefficients. Let $D$ be an effective Cartier divisor supported by $B$.
Let $f\colon X\to Z$ be a proper morphism. Then the natural homomorphism $R^if_* \cL \to R^if_* \cL(D)$ 
is injective, for every $i$.
\end{thm}

\begin{proof}
We may suppose $Z$ is affine.
Denote $\Sigma=\Supp B$ and $U=X\setminus \Sigma$. Since $rB$ is Cartier, we have 
an isomorphism 
$
\varinjlim_{m\in \N} H^i (X,\cO_X(mrB)) \isoto H^i (U,\cL|_U).
$
The claim for all 
$D$ is thus equivalent to the injectivity of the restriction homomorphisms
$$
H^i(X,\cL)\to H^i(U,\cL|_U).
$$
Let $\epsilon\colon X_\bullet\to X$ be the smooth simplicial resolution induced by the normalization of $X$.
Let $\Sigma_n=\epsilon_n^{-1}(\Sigma)$ and $U_n=X_n\setminus \Sigma_n$. The restriction 
$\epsilon\colon U_\bullet\to U$ is also a smooth simplicial resolution. 
By Lemma~\ref{slv2}, $\cL\to R\epsilon_*\cL_\bullet$ and $\cL|_U\to R\epsilon_*\cL_\bullet|_{U_\bullet}$ are
quasi-isomorphisms. Therefore the claim is equivalent to the injectivity of the restriction homomorphisms
$$
\alpha\colon H^i(X_\bullet,\cL_\bullet)\to H^i(U_\bullet,\cL_\bullet|_{U_\bullet}).
$$
Both spaces are endowed with simplicial filtrations $S$.
The Godement resolutions $\cL_p \to \cK^*_p\ (p\ge 0)$ glue to a simplicial resolution
$\cL_\bullet\to \cK^*_\bullet$. 
Denote $A^q_p=\Gamma_{\Sigma_p}(X_p,\cK_p^q)$, $B^q_p=\Gamma(X_p,\cK_p^q)$ and
$C^q_p=\Gamma(U_p,\cK_p^q)$. The associated simple complexes fit into a 
short exact sequence 
$$
0\to A\to B\to C\to 0
$$
which induces in homology the long exact sequence 
$$
\cdots \to H^i_{\Sigma_\bullet}(X_\bullet,\cL_\bullet)\to 
H^i (X_\bullet,\cL_\bullet)\to H^i(U_\bullet,\cL_\bullet|_{U_\bullet} )\to \cdots.
$$
Let $S$ be the simplicial filtration (naive with respect to $p$) on $A, B,C$.
For each $p$, the short exact sequence
$$
0\to \Gamma_{\Sigma_p}(X_p,\cK_p^*)\to \Gamma(X_p,\cK_p^*)\to \Gamma(U_p,\cK_p^*)\to 0
$$
is split. That is 
$
0\to E_0A\to E_0B\to E_0C\to 0
$
is a split short exact sequence. Passing to homology, we obtain that 
$
0\to E_1A\to E_1B\to E_1C\to 0
$
is a split short exact sequence. Iterating this argument,  we conclude that 
$
0\to E_rA\to E_rB\to E_rC\to 0
$
is a split short exact sequence, for every $r$. Therefore 
$
0\to E_\infty A\to E_\infty B\to E_\infty C\to 0
$
is a short exact sequence, which induces in homology the long exact sequence 
$$
\cdots \to \Gr_S H^i_{\Sigma_\bullet}(X_\bullet,\cL_\bullet)\to 
\Gr_S H^i (X_\bullet,\cL_\bullet)\to \Gr_S H^i(U_\bullet,\cL_\bullet|_{U_\bullet})\to \cdots.
$$

{\em Step 1}: $H^q_{\Sigma_p}(X_p,\cL_p)\to H^q(X_p,\cL_p)$ is zero for all $p,q$.
Indeed, 
$$
\cL_p^{\otimes 2r}=\cL^{\otimes 2r}|_{X_p}\simeq \omega^{[2r]}_{(X,B)}|_{X_p} \isoto \omega^{[2r]}_{(X_p,B_p)},
$$
$(X_p,B_p)$ is a log smooth variety, $U_p\supseteq X_p\setminus B_p$ by Lemma~\ref{slv2}.4), 
and $X_p\to Z$ is proper.
By Lemma~\ref{EVsmooth}, $H^q(X_p,\cL_p)\to H^q(U_p,\cL_p|_{U_p})$ is injective for all $p,q$.
Equivalently, 
$$
H^q_{\Sigma_p}(X_p,\cL_p)\to H^q(X_p,\cL_p)
$$ 
is zero for all $p,q$.

{\em Step 2}: $\Gr_S\alpha$ is injective. 
Indeed, $E_1A\to E_1B$ is the direct sum of $H^q_{\Sigma_p}(X_p,\cL_p)\to H^q(X_p,\cL_p)$. 
By Step 1, $E_1A\to E_1B$ is zero. Step by step, we deduce that $E_rA\to E_rB$ 
is zero for every $r\ge 1$. Then $E_\infty A\to E_\infty B$ is zero, that is 
$Gr_S H^i_{\Sigma_\bullet}(X_\bullet,\cL_\bullet)\to \Gr_S H^i (X_\bullet,\cL_\bullet)$ is zero.
Therefore the last long exact sequence breaks up into short exact sequences
$$
0 \to \Gr_S H^i (X_\bullet,\cL_\bullet)\to \Gr_S H^i(U_\bullet,\cL_\bullet|_{U_\bullet})\to 
\Gr_S H^{i+1}_{\Sigma_\bullet}(X_\bullet,\cL_\bullet) \to 0. 
$$

{\em Step 3}:  Since $S_{i+1}H^i(X_\bullet,L_\bullet)=0$, the filtration $S$ on $H^i(X_\bullet,L_\bullet)$
is finite. Therefore the injectivity of $\Gr_S \alpha$ 
means that $\alpha$ is injective and strict with respect to the filtration $S$.
\end{proof}

\begin{lem}\label{TKsmooth}
Let $(X,B)$ be a log smooth variety, let $f\colon X\to Z$ be a proper morphism.
Let $L$ be a Cartier divisor such that the $\Q$-divisor $A=L-(K_X+B)$ is $f$-semiample.
Let $D$ be an effective Cartier divisor on $X$ such that $D\sim_\Q uA$ for some $u>0$,
and $D$ contains no lc center of $(X,B)$.
Then the natural homomorphism $R^qf_*\cO_X(L)\to R^qf_*\cO_X(L+D)$ is injective, for all $q$.
\end{lem}

\begin{proof} We may suppose $Z$ is affine, and $A$ is $f$-semiample.

{\em Step 1}: Suppose $(X,B+\epsilon D)$ is log smooth, for some $0<\epsilon<\frac{1}{u}$.
We have 
$$
L=K_X+B+\epsilon D+(A-\epsilon D)\sim_\Q K_X+B+\epsilon D+(1-\epsilon u) A.
$$
Let $n\ge 1$ such that $\cO_X(nA)$ is generated by global sections. Let $S$ be the
zero locus of a generic global section. Then 
$$
L\sim_\Q K_X+B+\epsilon D+\frac{1-\epsilon u}{n} S,
$$
the log variety $(X,B+\epsilon D+\frac{1-\epsilon u}{n} S)$ is log smooth, and its boundary
supports $D$. By Lemma~\ref{EVsmooth}, $H^q(X,\cO_X(L)) \to H^q(X,\cO_X(L+D))$
is injective, for all $q$.

{\em Step 2}: By Hironaka, there exists a desingularization $\mu\colon X'\to X$ such that
the exceptional locus of $\mu$ and the proper transforms of $B$ and $D$ are supported
by a NC divisor on $X'$. Let $\mu^*(K_X+B)=K_{X'}+B_{X'}$, let $E=\lceil -B_{X'}^{<0}\rceil$.
Then 
$$
\mu^*L+E=K_{X'}+B_{X'}^{\ge 0}+\{B_{X'}^{<0} \}+\mu^*A.
$$
The log variety $(X',B_{X'}^{\ge 0}+\{B_{X'}^{<0} \}+\epsilon \mu^*D)$ is log smooth for $0<\epsilon\ll 1$,
by the choice of the resolution, and since $D$ contains no lc centers of $(X,B)$. We also have 
$\mu^*D\sim_\Q u \mu^*A$. By Step 1, the natural homomorphisms
$$
H^q(X',\cO_{X'}(\mu^*L+E)) \to H^q(X',\cO_{X'}(\mu^*L+E+\mu^*D))
$$
are injective. Consider now the commutative diagram
$$
\xymatrix{
H^q(X',\cO_{X'}(\mu^*L+E))  \ar[rr]^{\alpha'}  & &  H^q(X',\cO_{X'}(\mu^*L+E+\mu^*D))  \\
H^q(X,\cO_X(L))  \ar[rr]^\alpha  \ar[u]^\beta &  & H^q(X,\cO_X(L+D))   \ar[u]
}
$$
From above, $\alpha'$ is injective. If $\beta$ is injective, it follows that $\alpha$ is injective.
To show that $\beta$ is injective, suffices to show that $\cO_X\to R\mu_*\cO_{X'}(E)$
has a left inverse. The Cartier divisor $E'=K_{X'}-\mu^*K_X$ is effective, and $-B_{X'}\le E'$. 
Therefore $E\le E'$. We obtain homomorphisms
$$
\cO_X\to R\mu_*\cO_{X'}(E)\to R\mu_*\cO_{X'}(E').
$$
Suffices to show that the composition has a left inverse. Tensoring with $\omega_X$,
this is just the homomorphism $\omega_X\to R\mu_*\omega_{X'}$, which admits a left inverse
defined by trace (see the proof of~\cite[Proposition 4.3]{Del69}).
\end{proof}

\begin{thm}[Tankeev-Koll\'ar injectivity]\label{TKGNC}
Let $(X,B)$ be a GNC log variety, let $f\colon X\to Z$ be a proper morphism.
Let $\cL$ be an invertible $\cO_X$-module such that $\cL^{\otimes r}\simeq \omega^{[r]}_{(X,B)}\otimes \cH$,
where $r\ge 1$ and $rB$ has integer coefficients, and $\cH$ is an invertible $\cO_X$-module 
such that $f^*f_*\cH\to \cH$ is surjective. 
Let $s\in \Gamma(X,\cH)$ be a global section which is invertible at the generic point of each
lc center of $(X,B)$, let $D$ be the effective Cartier divisor defined by $s$. In particular, $D$ 
contains no lc center of $(X,B)$.
Then the natural homomorphism $R^qf_*\cO_X(L)\to R^qf_*\cO_X(L+D)$ is injective, for all $q$.
\end{thm}

\begin{proof} We may suppose $Z$ is affine. In particular, $\cH$ is generated by global sections.
Let $U=X\setminus \Supp D$. The claim for $D$ and all its multiples is equivalent to the
injectivity of the restriction homomorphisms $H^i(X,\cL)\to H^i(U,\cL|_U)$.

The proof is the same as that of Theorem~\ref{EVGNC}, except that in Step 1
we use Lemma~\ref{TKsmooth} instead of Lemma~\ref{EVsmooth}. Indeed, 
$\cL_p^{\otimes 2r}\isoto \omega^{[2r]}_{(X_p,B_p)}\otimes \cH^{\otimes 2}_p$, 
$(X_p,B_p)$ is log smooth, $\cH_p$ is generated by global sections, and 
$\epsilon_p^*D\in |\cH_p|$ contains no lc center of $(X_p,B_p)$. Therefore
$H^q(X_p,\cL_p)\to H^q(U_p,\cL_p|_{U_p})$ is injective, where $U_p=\epsilon_p^{-1}(U)=X_p\setminus \Supp \epsilon^*_pD$.
\end{proof}

\begin{thm}[Koll\'ar's torsion freeness]\label{TFGNC}
Let $(X,B)$ be a GNC log variety. Let $\cL$ be an invertible $\cO_X$-module such that 
$\cL^{\otimes r}\simeq \omega^{[r]}_{(X,B)}$ for some $r\ge 1$ such that $rB$ has integer 
coefficients. Let $f\colon X\to Z$ be a proper morphism. Let $s$ be a local section of $R^qf_*\cL$ 
whose support does not contain $f(C)$, for every lc center $C$ of $(X,B)$. Then $s=0$.
\end{thm}

\begin{proof} 
Suppose by contradiction that $s\ne 0$. Choose a closed point $P\in \Supp(s)$. We shrink
$Z$ to an affine neighborhood of $P$. There exists a non-zero divisor $h\in \cO_{Z,P}$ which 
vanishes on $\Supp(s)$, but does not vanish identically on $f(C)$, for every lc center $C$ of 
$(X,B)$. There exists $n\ge 1$ such that $h^ns=0$ in $(R^qf_*\cL)_P$. 

After shrinking $Z$ near $P$, 
we may suppose that $0\ne s\in \Gamma(Z,R^qf_*\cL)$, $h\in \Gamma(Z,\cO_Z)$
is a non-zero divisor, $h^ns=0$, and $h$ is invertible at the generic point of $f(C)$, for every lc center
$C$ of $(X,B)$. Since $Z$ is affine, we have an isomorphism $\Gamma(Z,R^qf_*\cL)\simeq H^q(X,\cL)$.
Therefore the multiplication $\otimes f^*h^n\colon H^q(X,\cL)\to H^q(X,\cL)$ is not injective. But $f^*h\in \Gamma(X,\cO_X)$
is invertible at the generic point of each lc center of $(X,B)$. By Theorem~\ref{TKGNC} with $\cH=\cO_X$, 
the multiplication $\otimes f^*h\colon H^q(X,\cL)\to H^q(X,\cL)$ is injective. Contradiction!
\end{proof}

\begin{thm}[Ohsawa-Koll\'ar vanishing]\label{OKGNC}
Let $(X,B)$ be a GNC log variety, let $f\colon X\to Y$ be a proper morphism and $g\colon Y \to Z$ a projective
morphism.
Let $\cL$ be an invertible $\cO_X$-module such that $\cL^{\otimes r} \simeq \omega^{[r]}_{(X,B)}\otimes  
f^*\cA$, where $r\ge 1$ and $rB$ has integer coefficients, and $\cA$ is a $g$-ample invertible $\cO_Y$-module. 
Then $R^pg_*R^qf_*\cL=0$ for all $p>0,q\ge 0$.
\end{thm}

\begin{proof} We use induction on the dimension of $X$.
We may suppose $Z$ is affine. Replacing $r$ by a multiple, we may suppose $\cA$ is 
$g$-generated. Let $m$ be a sufficiently large integer, to be chosen later.
Let $S$ be the zero locus of a general global section of $\cA^{\otimes m}$. Denote $T=f^*S$. 

Consider the short exact sequence
$$
0\to \cL\to \cL(T)\to \cL(T)|_T\to 0.
$$
The connecting homomorphism $\partial\colon R^qf_*\cL(T)|_T\to R^{q+1}f_*\cL$ is zero by Theorem~\ref{TFGNC}, 
since the image is supported by $T$, which contains no lc center of $(X,B)$, and 
$\cL^{\otimes r} \simeq \omega^{[r]}_{(X,B)}$ locally over $Y$. Therefore the long exact sequence
in cohomology breaks up into short exact sequences
$$
0\to R^qf_*\cL\to R^qf_*\cL(T)\to R^qf_*\cL(T)|_T\to 0.
$$

We have $R^pg_*R^qf_*\cL(T)\simeq R^pg_*(R^qf_*\cL(S))\simeq R^pg_*(R^qf_*\cL\otimes \cA^m)$.
If $m$ is sufficiently large, Serre vanishing gives $R^pg_*R^qf_*\cL(T)=0$ for $p\ne 0$.
By~\cite{TFR2},  $(X,B+T)$ is 
a GNC log variety, $T$ is $S_2$ and there exists a natural boundary $B_T=B|_T$ such that 
$(T,B_T)$ is a GNC log variety, and codimension one residues glue to residue isomorphisms
$$
\Res^{[2r]}_{X\to T}\colon \omega^{[2r]}_{(X,B+T)}|_T \isoto \omega^{[2r]}_{(T,B_T)}.
$$
From $\cL(T)^{\otimes r}\simeq \omega^{[r]}_{(X,B+T)}\otimes f^*\cA$ we obtain
$\cL(T)|_T^{\otimes 2r}\simeq  \omega^{[2r]}_{(T,B_T)}\otimes (f|_T)^*(\cA|_S^{\otimes 2})$. Since 
$\dim T<\dim X$, we obtain by induction $R^pg_*R^qf_*\cL(T)|_T=0$ for $p\ne 0$.

From the short exact sequence above, we deduce $R^pg_*R^qf_*\cL=0$ for $p\ge 2$.
For $p=1$, consider the commutative diagram
\[ 
\xymatrix{
   R^{1+q}(g\circ f)_*\cL  \ar[r]^\beta   &   R^{1+q}(g\circ f)_*\cL(T)    \\
   R^1g_*R^qf_*\cL  \ar[r] \ar[u]  &  R^1g_*R^qf_*\cL(T)=0  \ar[u]   
} \]
The vertical arrows are injective, from the Leray spectral sequence.
The homomorphism $\beta$ is injective by Theorem~\ref{TKGNC}, 
since $T\in |f^*\cA^{\otimes m}|$ contains no lc centers of $(X,B)$.
A diagram chase gives $R^1g_*R^qf_*\cL=0$.
\end{proof}


\section{Inductive properties of GNC log varieties}


\begin{prop}\label{lcsadj}
Let $(X,B)$ be a GNC log variety. Let $Y=\LCS(X,B)$ and $(Y,B_Y)$ the n-wlc structure 
induced by glueing of codimension one residues. Then $B_Y=(B-B^{=1})|_Y$ and 
$(Y,B_Y)$ is a GNC log variety. If $2\mid r$ and $rB$ has integer coefficients,
then $\Res^{[r]}\colon \omega^{[r]}_{(X,B)}|_Y\isoto \omega^{[r]}_{(Y,B_Y)}$ is an isomorphism. 
Moreover, 
\begin{itemize}
\item[1)] Let $\pi\colon (\bar{X},B_{\bar{X}})\to (X,B)$ be the normalization of $X$, with induced
log variety structure (with log smooth support). Let $\bar{Y}=\LCS(\bar{X},B_{\bar{X}})$. Let 
$n\colon Y^n\to Y$ and $\bar{n}\colon \bar{Y}^n\to \bar{Y}$ be the normalizations. In the 
commutative diagram
\[ 
\xymatrix{
 \bar{X}  \ar[d]_\pi  & \bar{Y} \ar[d]_\pi \ar[l] & & \bar{Y}^n \ar[d]^g \ar[ll]_{\bar{n}} &   \\
 X & Y \ar[l]   & &      Y^n \ar[ll]_{n}   
} \]
each square is both cartesian and a push-out, and $g$ is an \'etale covering. 
With the log structures induced by glueing of codimension one residues, we obtain a 
commutative diagram of GNC log varieties and log crepant morphisms
\[ 
\xymatrix{
 (\bar{X} ,B_{\bar{X}}) \ar[d]_\pi  & (\bar{Y},B_{\bar{Y}}) \ar[d]_\pi \ar[l] & & (\bar{Y}^n,B_{\bar{Y}^n} ) \ar[d]^g \ar[ll]_{\bar{n}} &   \\
 (X,B) & (Y,B_Y) \ar[l]   & &       (Y^n,B_{Y^n}) \ar[ll]_{n}   
} \]
\item[2)] The lc centers of $(X,B)$ are the irreducible components of $X$ and the 
lc centers of $(Y,B_Y)$.
\end{itemize}
\end{prop}

\begin{proof} 
1) We may suppose $(X,B)$ is a GNC local model.
Let $X=\cup_F\bA_F\hookrightarrow \bA^N$ and $B=\sum_{i\in \sigma}b_iH_i|_X$,
with core $\sigma=\cap_F F$. Denote $\sigma'=\{i\in \sigma;b_i<1\}$. Then $\psi=\sum_{i\in \sigma}(1-b_i)m_i=
\sum_{i\in \sigma'}(1-b_i)m_i$, which belongs to the relative interior of $\sigma'$. 
We deduce that $X_\gamma$ is an lc center of $(X,B)$ if and only if 
$\sigma'\prec \gamma\in \Delta$. Therefore $Y=\cup_\tau \bA_\tau\hookrightarrow \bA^N$ is an irreducible decomposition, 
where the union is taken after all codimension one faces $\tau\in \Delta$ which contain $\sigma'$. 
In particular, the core of $Y$ is $\sigma'$.  One checks that $(Y,0)$ satisfies properties a) and b) of the GNC 
local model. The boundary induced by codimension one residues is $B_Y=\sum_{i\in \sigma'}b_iH_i|_Y=(B-B^{=1})|_Y$, 
which satisfies c). The commutative diagram becomes
\[ 
\xymatrix{
 \sqcup_F \bA_F  \ar[d]_\pi  & \sqcup_F\cup_{\tau\prec F} \bA_\tau \ar[d]_\pi \ar[l] & & \sqcup_F\sqcup_{\tau\prec F} \bA_\tau \ar[d]^g \ar[ll]_{\bar{n}} &   \\
 \cup_F \bA_F & \cup_\tau \bA_\tau \ar[l]   & &      \sqcup_\tau \bA_\tau \ar[ll]_{n}   
} \]
and one checks that both squares are push-outs and cartesian, using axioms a) and b) of the GNC local models.
Over $\bA_\tau$, $g$ consists of several identical copies of $\bA_\tau$, one for each facet $F$ which contains $\tau$. Therefore $g$ is an \'etale covering.
All log structures have the same log discrepancy function $\psi$, hence the morphisms of the diagram are log crepant.

2) {\em Step 1}: The claim holds if $(X,B)$ is a GNC local model. Indeed, 
the lc centers of $(X,B)$ are the invariant cycles $X_\gamma$ such that 
$\psi\in \gamma$ and $\gamma\in \Delta$, and the lc centers of 
$(Y,B_Y)$ are the invariant cycles $X_\gamma$ such that 
$\psi\in \gamma$ and $\gamma\in \Delta$ is a face of positive codimension.

{\em Step 2}: We reduce the claim to the case when $(X,B)$ has log smooth support.
Indeed, consider the commutative diagram of log structures in 1).
The log structure on the normalization $(\bar{X}, B_{\bar{X}})$ has log smooth support. 
By Lemma~\ref{ilc} for $g$ and a diagram chase, the claim for $(X,B)$ and its $\LCS$-locus is equivalent to
the claim for $(\bar{X} ,B_{\bar{X}})$ and its $\LCS$-locus.

{\em Step 3}: Let $(X,B)$ have log smooth support. Then $Y=B^{=1}$ and the induced boundary 
is $B_Y=(B-Y)|_Y$. We have to show that for a closed subset $Z\subseteq Y$, $Z$ is an lc center
of $(X,B)$ if and only if $Z$ is an lc center of $(Y,B_Y)$, i.e. the image of an lc center of the normalization
$(Y^n,B_{Y^n})$. We may cut with general hyperplane sections, and suppose $Z$ is a closed point $P$.
Note that if $f\colon (X',B_{X'})\to (X,B)$ is \'etale log crepant, then $Y'=f^*Y$, and since normalization
commutes with \'etale base change, we obtain a cartesian diagram 
\[ 
\xymatrix{
 ({Y'}^n,B_{{Y'}^n}) \ar[d]_{n'}  \ar[r]^g & (Y^n,B_{Y^n})\ar[d]_n   \\
 (X',B_{X'}) \ar[r]^f  & (X,B)    
} \]
with $f,g$ \'etale log crepant. By Lemma~\ref{ilc} for $f$ and $g$, the claim holds for $n$ if and only if it holds for $n'$. 
By the existence of a common \'etale neighborhood~\cite{Ar69} and Step 1, we are done.
\end{proof}

\begin{cor}
Let $X$ be a GNC log variety. Then $S=\Sing X$ coincides with the non-normal locus of $X$, and
with $\LCS(X,0)$. The n-wlc structure induced by glueing of codimension one residues is 
$(S,0)$, a GNC log variety, and $\Res^{[2]}\colon \omega^{[2]}_X|_S\isoto \omega^{[2]}_{S}$ 
is an isomorphism. Moreover, 
\begin{itemize}
\item[1)] Let $\pi\colon (\bar{X},\bar{C})\to (X,0)$ be the normalization of $X$, with induced
log variety structure (with log smooth support). Note that $\bar{C}=\LCS(\bar{X},\bar{C})$. Let 
$n\colon S^n\to S$ and $\bar{n}\colon \bar{C}^n\to \bar{C}$ be the normalizations. In the 
commutative diagram
\[ 
\xymatrix{
 \bar{X}  \ar[d]_\pi  & \bar{C} \ar[d]_\pi \ar[l] & & \bar{C}^n \ar[d]^g \ar[ll]_{\bar{n}} &   \\
 X & S \ar[l]   & &      S^n \ar[ll]_{n}   
} \]
each square is both cartesian and a push-out, and $g$ is an \'etale covering. 
With the log structures induced by glueing of codimension one residues, we obtain a 
commutative diagram of GNC log varietis and log crepant morphisms
\[ 
\xymatrix{
 (\bar{X}, \bar{C}) \ar[d]_\pi  & (\bar{C},0) \ar[d]_\pi \ar[l] & & (\bar{C}^n,\Cond \bar{n} ) \ar[d]^g \ar[ll]_{\bar{n}} &   \\
 (X,0) & (S,0) \ar[l]   & &       (S^n,\Cond n) \ar[ll]_{n}   
} \]
\item[2)] The lc centers of $X$ are the irreducible components of $X$ and the lc centers of $S$.
\end{itemize}
\end{cor}

\begin{proof}
It remains to check that $S=C=\LCS(X,0)$. First of all, we claim that $S=C$. Indeed, let 
$x\in X$. We show that $\cO_{X,x}$ is normal if and only if $\cO_{X,x}$ is nonsingular.
We may suppose $x\in X$ is a local model $X=\cup_F X_F$ and $x$ belongs to the closed
orbit of $X$. Then $\cO_{X,x}$ is normal if and only if there is only one facet $F$. As 
$X_F$ is smooth, the latter is equivalent to $\cO_{X,x}$ being smooth.

Since $X\setminus S$ is smooth, $\LCS(X,0)\subseteq S$.
On the other hand, each irreducible component $Q$ of $S$ is an irreducible 
component of $C$. Therefore $Q$ is an lc center. We conclude that $\LCS(X,0)=S$.
\end{proof}

\begin{rem}
Let $(X,B)$ be a GNC log variety. Let $S=\Sing X$ and $B_S=B|_S$. One can also show
that $(S,B_S)$ is a GNC log variety, induced by codimension one residues. 
If $2\mid r$ and $rB$ has integer coefficients, the glueing of codimension one residues induces an isomorphism
$\Res^{[r]}\colon \omega^{[r]}_{(X,B)}|_S\isoto \omega^{[r]}_{(S,B_S)}$.
\end{rem}

\begin{lem}
Let $(X,B)$ be a GNC log variety. 
Let $\pi\colon (\bar{X} ,B_{\bar{X}})\to (X,B)$ be the normalization of $X$, with the induced log variety structure.
Let $Y=\LCS(X,B)$. Let $Z$ be a union of lc centers of $(X,B)$.
\begin{itemize}
\item[1)] $Z\cap Y$ is a union of lc centers of $(Y,B_Y)$.
\item[2)] $\pi^{-1}(Z)$ is a union of lc centers of $(\bar{X},B_{\bar{X}})$.
\item[3)] We have a short exact sequence
$
0\to \cI_{Z\cup Y\subset X} \to \cI_{Z\subset X} \stackrel{|_Y}{\to} \cI_{Z\cap Y\subset Y} \to 0.
$
\end{itemize}
\end{lem}

\begin{proof} 
1) We may suppose $Z$ is an lc center.
If $Z\subseteq Y$, the claim is clear. Therefore we may suppose $Z$ is an irreducible 
component of $X$. Then the normalization $\bar{Z}$ of $Z$ is an irreducible component
of the normalization $\bar{X}$ of $X$. We have $\pi^{-1}(Y)=\bar{Y}=\LCS(\bar{X},B_{\bar{X}})$.
Therefore $Z\cap Y=\pi(\bar{Z}\cap \bar{Y})$. We have $\bar{Z}\cap \bar{Y}=\LCS(\bar{X},B_{\bar{X}})|_{\bar{Z}}$,
we deduce that $\bar{Z}\cap \bar{Y}$ is a union of lc centers of $(\bar{X},B_{\bar{X}})$ contained in
$\bar{Y}$. Therefore $Z\cap Y$ is a union of lc centers of $(X,B)$ contained in $Y$, hence lc
centers of $(Y,B_Y)$, by Proposition~\ref{lcsadj}.

2) We use induction on $\dim X$. We may suppose $Z$ is an lc center.
If $Z$ is an irreducible component of $X$, then its normalization $\bar{Z}$ is 
an irreducible component of $\bar{X}$, and $\pi^{-1}(Z)=\bar{Z}\cup \pi^{-1}(Z\cap Y)$,
since $Y$ contains the non-normal locus of $X$. By induction, the claim holds.

Suppose $Z$ is not an irreducible component of $X$. Then $Z\subseteq Y$, by Proposition~\ref{lcsadj}.
By induction, $n^{-1}(Z)$ is a union of lc centers of $(Y^n,B_{Y^n})$. Let $W$ be such an lc center.
Since $g$ is finite flat, each irreducible component of $g^{-1}(W)$ dominates $W$. Therefore 
$g^{-1}n^{-1}(Z)$ is a union of lc centers of $(\bar{Y}^n,B_{\bar{Y}^n} )$, by Lemma~\ref{ilc}.
Equivalently, $\bar{n}^{-1}\pi^{-1}(Z)$ is a union of lc centers of $(\bar{Y}^n,B_{\bar{Y}^n} )$.
Therefore $\pi^{-1}(Z)$ is a union of lc centers of $(\bar{Y},B_{\bar{Y}})$. The latter lc centers 
are also lc centers of $(\bar{X} ,B_{\bar{X}})$. 

3) The sequence is exact if and only if $|_Y\colon \cI_{Z\subset X} \to \cI_{Z\cap Y\subset Y}$ is surjective,
if and only if $|_Y\colon \cI_{Z\subset Z\cup Y} \to \cI_{Z\cap Y\subset Y}$ is surjective,
if and only if the diagram 
\[ 
\xymatrix{
  Y  \ar[d]  &  Y\cap Z  \ar[d] \ar[l]    \\
  Y\cup Z         &    Z \ar[l]   
} \]
is a push-out. By~\cite{LV81}, this diagram is a push-out if $Y\cup Z$ is weakly normal.
To show this, consider the normalization $\pi\colon \bar{X}\to X$. Denote $W=\pi^{-1}(Y\cup Z)$.
Since 
\[ 
\xymatrix{
  \bar{X}  \ar[d]  &  \bar{Y}  \ar[d] \ar[l]    \\
  X         &    Y \ar[l]   
} \]
is a push-out and $Y\cup Z$ contains $Y$, the diagram
\[ 
\xymatrix{
  \bar{X}  \ar[d]  &  W  \ar[d] \ar[l]    \\
  X         &    Y\cup Z     \ar[l]   
} \]
is also a push-out.
But $\bar{X}$ is smooth, and $W$ is the union of $\bar{Y}$ with the irreducible components of 
$\bar{X}$ which are mapped into $Z$. Therefore the singularities of $W$ are at most normal crossings.
We conclude that $X,\bar{X},W$ are weakly normal. From the last push-out diagram, we 
deduce that $Y\cup Z$ is weakly normal as well.
\end{proof}

The results of this section can be used to reduce Koll\'ar's torsion freeness theorem and 
Ohsawa-Koll\'ar vanishing theorem from the GNC varieties to log smooth varieties. This 
is done by a using the push-out and cartesian diagram obtained from normalization and 
restriction to the $\LCS$-locus. We were unable to use the same argument to reduce the injectivity 
theorems from GNC varieties to log smooth varieties, but we expect this is possible.



\begin{thebibliography}{KKS10}


\bibitem{Amb03}
Ambro F.,
{\em Quasi-log varieties}, in 
{\em Birational Geometry: Linear systems and finitely generated 
algebras: Collected papers.}
{Iskovskikh, V.A. and Shokurov, V.V. (Ed.), Proc. V.A. Steklov 
Inst. Math. 240 (2003), 220--239.}

\bibitem{Amb14}
Ambro, F.,
{\em An injectivity theorem,}
{Compos. Math. {\bf 150}(6) (2014),  999 --1023.}

\bibitem{TFR1}
Ambro, F.,
{\em On toric face rings I.}
{preprint arXiv:1705.02759.}

\bibitem{TFR2}
Ambro, F.,
{\em On toric face rings II.}
{preprint arXiv:1705.02760.}

\bibitem{Ar69}
Artin, M.,
{\em Algebraic approximation of structures over complete local rings.}
{Inst. Hautes \'Etudes Sci. Publ. Math. {\bf 36} (1969), 23 -- 58.}

\bibitem{Del69}
Deligne, P., 
{\em Th\'eor\`eme de Lefschetz et crit\`eres de d\'eg\'en\'erescence 
de suites spectrales.}
{Publ. Math. IHES 35 (1969), pp. 107Ð-126.}

\bibitem{Kaw85}
Kawamata Y. 
{\em Pluricanonical systems on minimal algebraic varieties.}
{Invent. Math. {\bf 79} (3) (1985), 567--588.}

\bibitem{Kbook}
Koll\'ar, J.,
{\em Singularities of the Minimal Model Program.}
{Cambridge Tracts in Mathematics 200 (2013).}

\bibitem{LV81} 
Leahy J.V., Vitulli M.A., 
{\em Seminormal rings and weakly normal varieties},
{Nagoya Math. J. {\bf 82} (1981), 27-- 56.}

\end{thebibliography}
\end{document}